\documentclass[12pt]{amsart}
\usepackage{amsmath,enumerate,eqnarray}
\usepackage{amssymb,color}
\usepackage{amsthm}
\usepackage[utf8x]{inputenc}
\usepackage[pdftex]{graphicx}

\theoremstyle{plain}
\newtheorem{theorem}{Theorem}[section]
\newtheorem{cor}[theorem]{Corollary}
\newtheorem{lem}[theorem]{Lemma}

\newtheorem{prop}[theorem]{Proposition}

\theoremstyle{definition}
\newtheorem{defn}[theorem]{Definition}

\newcommand{\tl}{\triangleleft}
\newcommand{\tr}{\triangleright}
\newcommand{\twc}{\ast}
\newcommand{\s}{\underline{s}}
\newcommand{\Ss}{\underline{S}}
\newcommand{\twist}{\mathcal{I}}
\newcommand{\invol}{I_{2n}}
\newcommand{\Sn}{\mathcal{S}_n}
\newcommand{\iotint}[2]{[#1,#2]_\iota}
\newcommand{\twint}[2]{[#1,#2]_\mathcal{I}}
\newcommand{\trp}[2]{(#1\,\,\,#2)}
\newcommand{\br}{\mathrm{Br}}
\newcommand{\R}{\mathbb{R}}
\newcommand{\C}{\mathbb{C}}
\newcommand{\DR}{D_{\mathrm{R}}}
\newcommand{\SL}{\mathrm{SL}(2n,\C)}
\newcommand{\Sp}{\mathrm{Sp}(2n,\C)}
\newcommand{\KLP}{\mathcal{P}}

\title[Combinatorial invariance of KLV polynomials]{Combinatorial invariance of Kazhdan-Lusztig-Vogan polynomials
  for fixed point free involutions}
\author[Nancy Abdallah]{Nancy
  Abdallah}
\author[Axel Hultman]{Axel Hultman}

\address{Department of Mathematics, Link\"oping University, SE-581 83 Link\"oping, Sweden.}
\email{nancy.abdallah@liu.se, axel.hultman@liu.se}

\begin{document}
\begin{abstract}
When $\Sp$ acts on the flag variety of
$\SL$, the orbits are in bijection with fixed point free
involutions in the symmetric group $S_{2n}$. In this case, the associated
Kazhdan-Lusztig-Vogan polynomials $P_{v,u}$ can be indexed by pairs of fixed point
free involutions $v\ge u$, where $\ge$ denotes the Bruhat order on
$S_{2n}$. We prove that these polynomials are combinatorial invariants in the
sense that if $f:[u,w_0]\to [u',w_0]$ is a poset isomorphism of upper
intervals in the Bruhat order on fixed point free involutions, then
$P_{v,u} = P_{f(v),u'}$ for all $v\ge u$.
\end{abstract}

\maketitle

\section{Introduction}
Let $(W,S)$ be a Coxeter system with $W$ ordered by the Bruhat order. Every interval $[u,v]$ in
this poset comes with an associated Kazhdan-Lusztig (KL) polynomial $\KLP_{u,v}$,
introduced in \cite{KL}. When
$W$ is a Weyl group, the polynomials carry detailed information about
the singularities of the Schubert varieties indexed by $W$ \cite{KL2}. Evaluated
at $1$, they provide composition factor multiplicities of Verma
modules; this is one of the original Kazhdan-Lusztig
conjectures from \cite{KL} which was independently proven by
Beilinson and Bernstein \cite{BeBe} and by Brylinski and Kashiwara
\cite{BK}. 

The KL polynomial $\KLP_{u,v}$ can be
computed merely in terms of the structure of the lower interval
$[e,v]$, where $e\in W$ is the identity element, which is the minimum
in the Bruhat order. The procedure relies on detailed knowledge
about the elements of $[e,v]$. It has, however, been conjectured
independently by Dyer \cite{MD} and Lusztig that the KL polynomial is an
invariant of the poset isomorphism class of $[u,v]$. This
is known as the combinatorial invariance conjecture. Most
substantial progress towards this conjecture has had to do with lower
intervals and is captured in the following statement: 
\begin{theorem}\label{th:Winvariance}
Let $v,v'\in W$ and suppose $f:[e,v]\to[e,v']$ is a poset
isomorphism. Then, $\KLP_{u,v} = \KLP_{f(u),v'}$ for all $u\le v$.
\end{theorem}
 For $W$ of general type, this result is due to Brenti,
 Caselli and Marietti \cite{BCM} and independently to Delanoy
 \cite{ED}. The methods build on earlier work by Brenti \cite{FB} and
 du Cloux \cite{FdC}, where the result was established in certain types. 

The more general family of Kazhdan-Lusztig-Vogan (KLV) polynomials was
introduced in \cite{LV,DV}. Let $G$ be a complex connected reductive
algebraic group with non-compact real form $G_\R$. Let $\theta:G\to G$
be the complexification of a Cartan involution of $G_\R$. The fixed
point subgroup $K=G^\theta$ acts on the flag variety $G/B$ with
finitely many orbits \cite{TM}. Each of the indices $u$ and $v$ of a
KLV polynomial $P_{u,v}$ consists of a $K$-orbit closure together with a
choice of local system on it. These polynomials describe the singularities
of $K$-orbit closures and, evaluated at $1$, provide
character coefficients for $G_\R$-representations \cite{LV,DV}.

In this paper, we confine ourselves to the setting $G=\SL$, $G_\R =
\mathrm{SU}^*(2n)$, $K=\Sp$. The following
combinatorially appealing situation then arises: 
we may think of the indices $u$ and $v$ simply as fixed point free involutions in
the symmetric group $S_{2n}$ of permutations of
$\{1,\ldots,2n\}$. Moreover, $P_{u,v}$ is
nonzero if and only if $u\ge v$ in the Bruhat order on
$S_{2n}$. Denote by $\br(F_{2n})$ its subposet induced by the fixed
point free involutions. The maximum of $\br(F_{2n})$ is the reverse
permutation $w_0$, the longest element in $S_{2n}$. Our main result
is the following combinatorial invariance assertion for KLV polynomials:
\begin{theorem}\label{th:mainv1}
  If $f: [u,w_0] \to [u',w_0]$ is a poset isomorphism of upper
  intervals in $\br(F_{2n})$, then $P_{v,u}= P_{f(v),u'}$ for all fixed point free involutions $v\ge u$.
\end{theorem}

The fixed point free involutions $u\in S_{2n}$ which satisfy
$u(i)>n$ for all $i\le n$ form a subposet of $\br(F_{2n})$ which is
isomorphic to the dual of the Bruhat order on $S_n$. When restricted
to such $u$, Theorem \ref{th:mainv1} specialises to the type $A$
version of Theorem \ref{th:Winvariance}, which is the main result
of Brenti's aforementioned work \cite{FB}.

In order to briefly outline the proof idea, let us first describe
Brenti's approach from \cite{FB}. First, Brenti observed that
combinatorial invariance of KL
polynomials is equivalent to that of the associated
KL $R$-polynomials. At the heart of the recurrence relation for the
$R$-polynomials is the map $x\mapsto xs$ for $x\in W$, $s\in
S$. Brenti's key idea was to replace such maps by {\em special
matchings}, which are defined solely in terms of poset properties. By
studying the possible special matchings of lower Bruhat intervals,
Brenti was able to deduce the key fact, namely that the resulting poset theoretic
recurrence actually is well-defined and computes the $R$-polynomials.

Our overall approach is very similar to that of Brenti. Instead of
KL $R$-polynomials, we study what we call $Q$-polynomials, which are a
slight variation of Vogan's KLV $R$-polynomials. In our setting, Vogan's recurrence for the
latter boils down to a recurrence for $Q$-polynomials which relies on
the conjugation map $x\mapsto sxs$ of fixed point free involutions
$x$. In order to obtain a poset theoretic recurrence, we replace
such conjugation maps by {\em special partial matchings} which are similar
to special matchings, except that they may have fixed elements. Again,
the crux is to show that this indeed yields a recurrence which
computes $Q$-polynomials. Rather than fixed point free
involutions, we mostly work with the set of twisted identities $\iota
\subset S_{2n}$; multiplication by $w_0$ provides a bijection between
the two which reverses the Bruhat order. This viewpoint gives us
convenient access to combinatorial tools already developed for
twisted identities.

The remainder of the paper is structured as follows. In the next
section, we agree on notation and recall important definitions and tools. Section \ref{se:iota} contains some
observations about twisted identities. Special
partial matchings are introduced in Section \ref{se:SPM}, where
technical assertions about the structure of such partial matchings
are collected. In the final section we use them in order to
prove our main result.

\section*{Acknowledgement}
N.\ Abdallah was funded by a stipend from Wenner-Gren Foundations.

\section{Preliminaries}
Let $n$ be a positive integer and denote by $W = S_{2n}$ the symmetric group of permutations of the set
$[2n] = \{1,\ldots,2n\}$. Then, $W$ is a Coxeter group with set of
Coxeter generators $S=\{s_1,\ldots,s_{2n-1}\}$, where the $s_i =
\trp{i}{i+1}$ are the adjacent transpositions. If $w= s_{i_1}\cdots
s_{i_k}$, the word $s_{i_1}\cdots s_{i_k}$ is an {\em expression} for
$w$ which is {\em reduced} if $k$ is minimal; then $k = \ell(w)$ is the
number of {\em inversions} of $w$, i.e.\
the number of pairs $1\le i < j \le 2n$ such that $w(i)>w(j)$.

A generator $s\in S$ is called a (right) {\em descent} of $w\in W$ if
$\ell(ws)<\ell(w)$. The set of descents of $w$ is denoted by
$\DR(w)$. Clearly, $s_i\in \DR(w)$ if and only if $w(i)>w(i+1)$.

\subsection{Twisted involutions and twisted identities}

Define an involutive
automorphism $\theta: W\to W$ by $\theta(s_i)=s_{2n-i}$. This is the only nontrivial (if $n>1$) automorphism of
$W$ which preserves $S$. 

Let $\twist = \twist(\theta) = \{w\in W\ \mid \theta(w)=w^{-1}\}$ be
the set of {\em twisted involutions} and
$\iota = \iota(\theta)=\{\theta(w^{-1})w\ \mid  w\in W \}\subset \twist$
the subset of {\em twisted identities}. In other words, $\iota$ is the orbit
of the identity element $e$ when $W$ acts (from the right, say) on
itself by twisted conjugation. Let $\twc$ denote this action; i.e.\
$x\twc w = \theta(w^{-1})xw$ for $x,w\in W$.

Next we recall some properties of $\twist$ and $\iota$. All
unjustified claims can be gleaned from \cite{RS} or \cite{AH1}. Our
notation follows the latter reference. Define a set of symbols $\Ss=\{\s_i \mid i \in
[2n-1]\}$. There is an action of the free monoid $\Ss^*$ on
the set $W$ defined by 
\[ 
w\s=
\begin{cases} 
w s & \text{ if } w \twc s =w, \\
w \twc s & \text{ otherwise.} 
\end{cases} 
\]
It is convenient to use the notational conventions $w\s_{i_1}\cdots
\s_{i_k}=(\cdots((w\s_{i_1})\s_{i_2})\cdots)\s_{i_k}$ and $\s_{i_1}\cdots\s_{i_k} =
e\s_{i_1}\cdots\s_{i_k}$, where $e$ is the identity permutation. The orbit of $e$ under this action is
$\twist$. Thus, if $w\in \twist$, we have $w = \s_{i_1}\cdots\s_{i_k}$
for some $i_j$. We refer to the word $\s_{i_1}\cdots\s_{i_k}$
as an {\em $\Ss$-expression} for $w$ and say it is {\em reduced}
if $k$ is minimally chosen among all such expressions; in that case $\rho(w)
= k$ is called the {\em rank} of $w$. If $w\in \iota$, $2\rho(w) =
\ell(w)$.

Just as with ordinary expressions, the Coxeter relations can be
applied to reduced $\Ss$-expressions: $\cdots \s_i\s_{i+1}\s_i\cdots =
\cdots \s_{i+1}\s_i\s_{i+1}\cdots$ and $\cdots \s_i \s_j\cdots = \cdots \s_j\s_i\cdots$
if $|i-j|> 2$. Unlike for ordinary expressions this is not in
general the case for arbitrary $\Ss$-expressions. For
example, with $n=2$,
\[
\s_2\s_3\s_2\s_1\s_2 = \s_2\s_1\s_3 = s_1s_3s_2s_1s_3 = 4231
\]
which is different from 
\[
\s_2\s_3\s_1\s_2\s_1 = \s_1\s_2\s_3 = s_2s_3s_1s_2s_3 = 3421.
\]

A useful property of $\iota$ is that $w\s \in \iota$ holds whenever
$w\in \iota$, $s\in \DR(w)$. Thus, if $w\twc s = w$ for $w\in \iota$,
$s\in S$, then $s\not \in \DR(w)$.

\subsection{The Bruhat order}

When applied to elements of $W$, $\le$ denotes the {\em Bruhat order}. We shall make use of several
well-known characterisations which all can be found in \cite{BB}:
\begin{theorem}[Subword Property of $W$]\label{th:origsubword}
Let $x,y\in W$, and suppose $s_{i_1}\cdots s_{i_k}$ is a reduced
expression for $y$. Then, $x\le y$ holds if and only if $x =
s_{i_{j_1}}\cdots s_{i_{j_l}}$ for some $1\le j_1<\cdots<j_l\le k$. 
\end{theorem}

For $x\in W$, let $x_{i,k}$ denote the $i$-th element when $x(1),
x(2), \ldots, x(k)$ are rearranged increasingly.

\begin{theorem}[Tableau Criterion]\label{crit}
	Given $x,y\in W$, the following conditions are equivalent:
\begin{enumerate}[(i)]
	\item $x \leq y$.
	\item $x_{i,k} \leq y_{i,k}$ for all $s_k \in \DR(x)$ and $1 \leq i \leq k$.
	\item $x_{i,k} \leq y_{i,k}$ for all $s_k \in S \setminus \DR (y)$ and $1 \leq i \leq k$.
\end{enumerate}	
\end{theorem}
\begin{theorem}\label{th:stdcrit}
Given $x,y\in W$, $x\le y$ holds if and only if $x[i,j] \le y[i,j]$ for all
$i,j\in[2n]$, where $w[i,j] = |\{k\in \{i,\ldots,2n\}\mid w(k)\le j\}|$.
\end{theorem}

The Bruhat order on $W$ is a graded poset with minimum element the 
identity permutation $e$, maximum element the reverse permutation
$w_0$ defined by $w_0(i)= 2n+1-i$, and rank function $\ell$. For any subset $X\subseteq
W$, $\br(X)$ indicates the subposet of the Bruhat order induced by
$X$. Obviously, the previous three results may be applied in the
context of any $\br(X)$ just by replacing $W$ with $X$. We next
discuss some additional tools which are applicable when $X=\twist$. Because $W$
is a Weyl group, most are consequences of Richardson and Springer's
\cite[Section 8]{RS}, although the terminology there differs
somewhat from that employed here. In our language, everything can be found
in \cite{AH1}.

The following counterpart of Theorem \ref{th:origsubword} can be
deduced from \cite[Corollary 8.10]{RS}:

\begin{theorem}[Subword Property of $\twist$]\label{th:twsubword}
Let $x,y\in \twist$, and suppose $\s_{i_1}\cdots \s_{i_k}$ is a
reduced $\Ss$-expression for $y$. Then, $x\le y$ holds if and only
if $x = \s_{i_{j_1}}\cdots \s_{i_{j_l}}$ for some $1\le j_1<\cdots<j_l\le k$.
\end{theorem}

Comparing Theorem \ref{th:origsubword} with Theorem \ref{th:twsubword} it
should not be too surprising that $\br(\twist)$ is graded with rank
function $\rho$. This is in fact also true for $\br(\iota)$, where
$\rho = \ell/2$. Both have
$e$ as minimum element. The maximum in $\br(\twist)$ is $w_0$, whereas the
maximum in $\br(\iota)$ is $w_0s_1s_3\cdots s_{2n-1}$. 

Let $\invol$ be the set of (ordinary) involutions in 
$W$. Incitti \cite{FI} showed that $\br(\invol)$ is Eulerian. In particular it and its dual $\br(\twist)$
are {\em thin}, meaning that all rank two intervals consist of
exactly four elements. Rank two intervals in the {\em subthin}
poset $\br(\iota)$ have either four or three elements.

The so called {\em lifting property} is a classical result on
$\br(W)$; see \cite[Theorem 1.1]{VD}. We shall not make explicit use of
it, but instead of the following completely analogous result for
$\br(\twist)$ which follows from \cite[Proposition 8.13]{RS}:
\begin{lem}[Lifting Property of $\twist$]\label{lift}
	Let $u,w\in \twist$ with $u\leq w$ and suppose $s\in \DR(w)$. Then,
	\begin{enumerate}[(i)]
		\item $u\s\leq w$.
		\item $s\in \DR(u)\Rightarrow u\s\leq w \s$.
		\item $s\notin \DR(u)\Rightarrow u\leq w \s$.
	\end{enumerate}
\end{lem}

Consider the following subset of $\iota$:
\[
\Sn = \{w\in \iota \mid \s_n \not \le w\}.
\]
Every element in $\Sn$ has a reduced $\Ss$-expression which consists
entirely of letters in $\{\s_1,\ldots,\s_{n-1}\}$. Given such an
expression, one obtains a reduced expression for an element in
$S_n$ by removing the lines under the letters: $\s_{i_1}\cdots
\s_{i_k}\mapsto s_{i_1}\cdots s_{i_k}$. This map is a bijection
$\varphi: \Sn
\to S_n$. It preserves the Bruhat order, yielding $\br(\Sn)\cong
\br(S_n)$. When restricted to $\Sn$, the main results of the present paper
coincide with Brenti's results on the symmetric group \cite{FB}.
  
\subsection{Permutation diagrams, symmetries and cover relations}

The inversion map $x\mapsto x^{-1}$ is a poset automorphism of
$\br(W)$ and, since it preserves $\iota$, of $\br(\iota)$. Left and
right multiplication by the reverse permutation $w_0$ yield 
poset antiautomorphisms of $\br(W)$. Composing them, we
recover the automorphism $\theta$; i.e.\ $\theta(w) = w_0ww_0$. Therefore,
$x\mapsto w_0x$ and $x\mapsto xw_0$ provide poset isomorphisms
$\br(\twist)\cong \br(\invol)^*$, where $P^*$ denotes the dual poset of $P$. Under both isomorphisms, $\iota$ is sent to the conjugacy class of $w_0$, namely the set $F_{2n}\subset W$ of fixed point free involutions, so that
$\br(\iota)\cong \br(F_{2n})^*$.

We shall sometimes represent $w\in W$ by means of its {\em
diagram}, i.e.\ the graph of $w$. It has a dot in the
integer point $(i,j)$ whenever $w(i) = j$. Theorem \ref{th:stdcrit}
can then be interpreted as follows: $x\le y$ iff for every
$(i,j)\in[2n]^2$, there are at least as many dots weakly
southeast of $(i,j)$ in the diagram of $y$ as there are in the diagram
of $x$.\footnote{It is equivalent, and probably more common, to replace ``southeast'' by
  ``northwest'' in this statement, since $180^\circ$ rotation of
  diagrams coincides with the Bruhat order automorphism $\theta$.}

Left multiplication by $w_0$ amounts to an upside down flip of the
diagram, whereas taking inverses is reflection in the diagonal line
through $(1,1)$ and $(2n,2n)$. It follows that $\twist$ consists of the permutations whose
diagrams are invariant under reflection in the line through
$(1,2n)$ and $(2n,1)$, and that $\iota$ is the subset of elements without any dots
on this line.

Two dots in a permutation diagram form a {\em rise} if the rightmost
dot is also the uppermost; otherwise the dots form a {\em fall}.
 
We shall reserve the notation $u\tl w$ to mean that $u$ is covered by $w$ in
$\br(\twist)$ (hence in $\br(\iota)$ if $u,w\in \iota$). In \cite{FI}, Incitti characterised the cover relation of
$\br(\invol)$ in terms of the diagrams of the involved
involutions. By taking duals and/or restricting, we obtain for free the cover relations in
$\br(\twist)$, $\br(\iota)$ and $\br(F_{2n})$. We reproduce Incitti's description in Figure
\ref{fi:incitti}, adapted to the setting of $\br(\twist)$. Observe that only two
of the six kinds of covers, namely those without dots on the diagonal,
occur in $\br(\iota)$. In particular, every cover in $\br(\iota)$ (respectively, $\br(F_{2n})$) is given by
twisted (respectively, ordinary) conjugation by a transposition. That
is, if $u\tl w$ and $u,w\in \iota$, then $u = w\twc t$
(respectively, $w_0u = tw_0wt$) for some transposition $t$.

\begin{figure}[htb]
\includegraphics{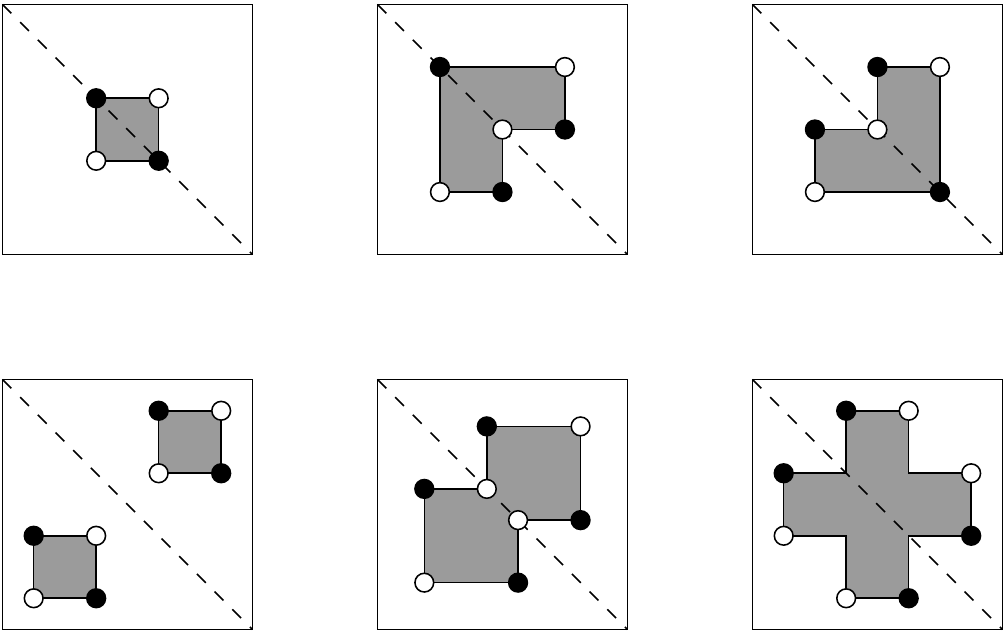}
\caption{All types of covers that occur in $\br(\twist)$. If $w$
  covers $u$, the diagram of $u$ is indicated by white dots and $w$ is
represented by black dots. Dots shared by both diagrams are
omitted. Shaded regions are empty. The pictures are reproduced from \cite{FI}.}\label{fi:incitti}
\end{figure}

By inspecting Figure \ref{fi:incitti}, the next lemma follows
immediately.

\begin{lem}\label{le:twistiotacover}
Given $w\in \twist \setminus \iota$, there exists at most one $u\in \iota$ such that
$u\tl w$.
\end{lem} 

Most of the action of the present paper takes place in $\br(\iota)$. However, other
subposets of $\br(W)$ turn up frequently in our arguments. In order to
mitigate possible confusion we shall employ the following poset interval notation for
$u,w\in W$:
\begin{eqnarray*}
[u,w] &=& \{x\in W\mid u\le x \le w\},\\
\twint{u}{w} &=& [u,w]\cap \twist,\\
\iotint{u}{w} &=& [u,w]\cap \iota.\\
\end{eqnarray*} 
Some examples can be found in Figure \ref{fi:posets}.
\begin{figure}[thb]
\includegraphics[scale=0.55]{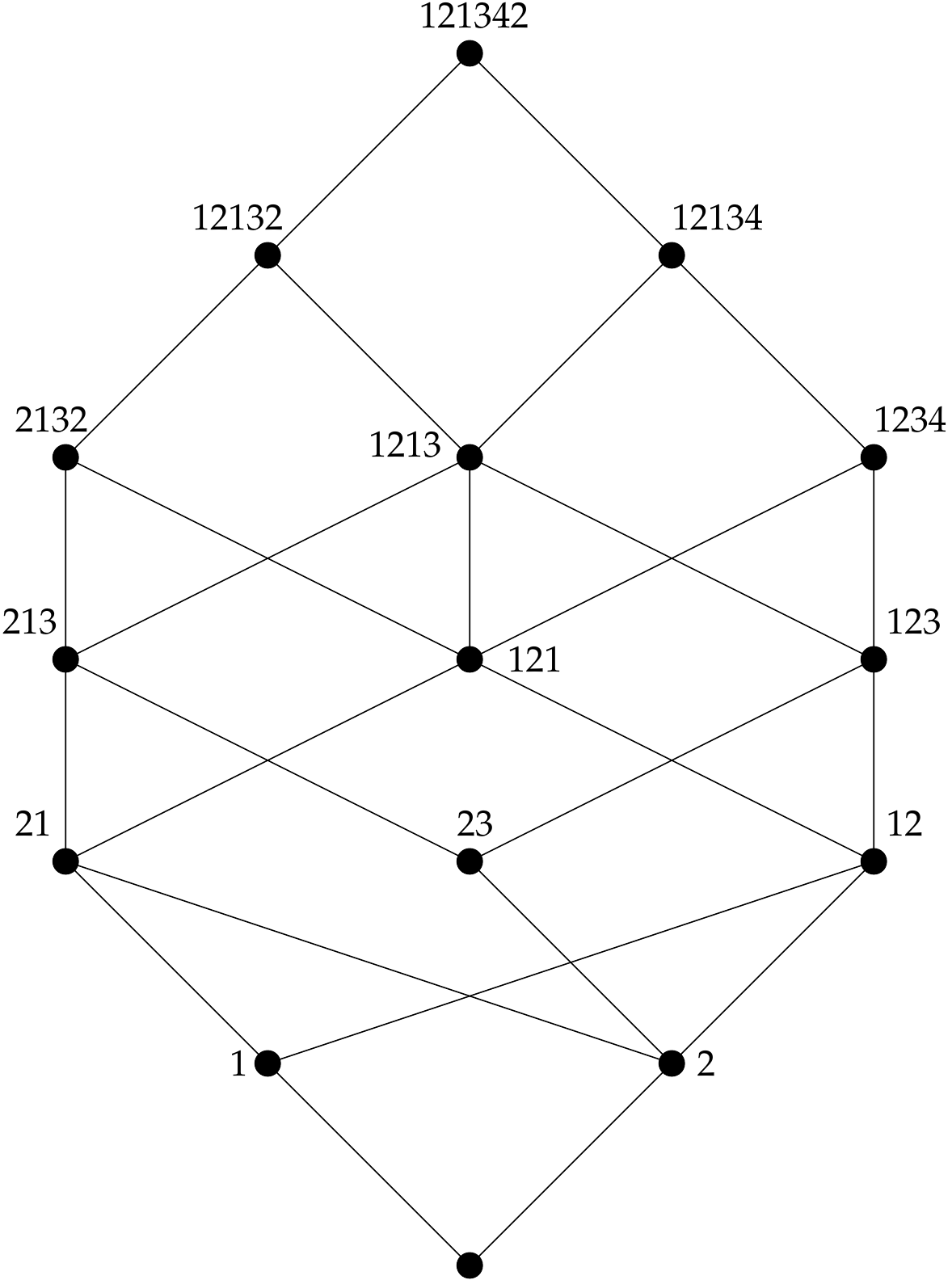}
\includegraphics[scale=0.55]{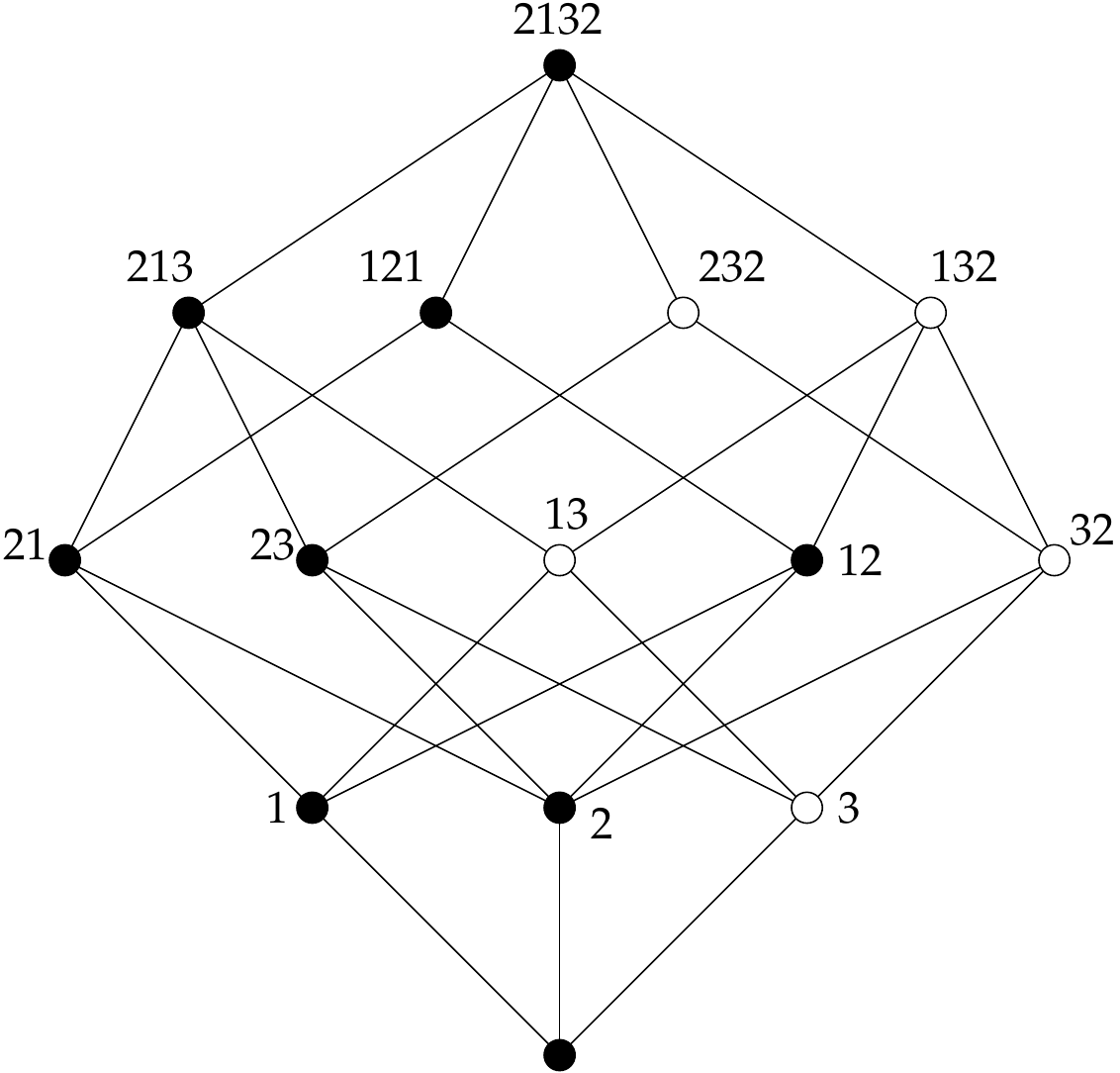}
\caption{Pictures of $\br(\iota) = \iotint{e}{\s_1\s_2\s_1\s_3\s_4\s_2}$
  (left) and $\twint{e}{\s_2\s_1\s_3\s_2}$ (right) when
  $n=3$. Twisted identities correspond to black dots, whereas white dots
  signal elements of $\twist\setminus \iota$. The labels are index
  sequences of reduced $\Ss$-expressions. For example, ``$232$''
  represents the twisted involution $\s_2\s_3\s_2 = s_3s_4s_2s_3s_2$.}  \label{fi:posets}
\end{figure}

\subsection{Kazhdan-Lusztig-Vogan polynomials} 

Introduced in \cite{LV,DV}, the Kazhdan-Lusztig-Vogan (KLV)
polynomials are at the heart of the representation theory of real
reductive groups much in the same way that Kazhdan-Lusztig polynomials
describe representations for complex groups.

In general, a KLV polynomial $P_{\gamma, \delta}(q)$ is indexed by two
local systems $\gamma$ and $\delta$ on orbits of a symmetric subgroup
$K$ on a flag manifold $G/B$. In the present paper we shall restrict
to the setting $G = \SL$,
$K=\Sp$. In this case every local system is
trivial and the orbits are indexed by $\iota$ (or, as was done in the
introduction, by $F_{2n}$; hence the title of the present
paper). Moreover, $\br(\iota)$
coincides with the inclusion order among orbit closures; the details
of this correspondence are described by Richardson and Springer in
\cite[Example 10.4]{RS}. Thus, we may in this setting consider KLV polynomials to be
indexed by pairs of twisted identities. When doing so, we shall use
the superscript $\iota$ to avoid confusion with the polynomials
indexed by $F_{2n}$ in the
introduction. In other words, $P^\iota_{u,w} = P_{w_0u,w_0w}$ whenever
$u,w\in \iota$. For fixed $w\in \iota$, we
then have the following identity in the free
$\mathbb{Z}[q,q^{-1}]$-module with basis $\iota$:
\begin{equation}\label{eq:KLV1}
q^{-\rho(w)}\sum_{v\in \iotint{e}{w}}P^\iota_{v,w}(q)v = \sum_{v\in
  \iotint{e}{w}}\sum_{u\in \iotint{e}{v}}(-1)^{\rho(u)-\rho(v)}q^{-\rho(v)}P^\iota_{v,w}(q^{-1})R_{u,v}(q)u;
\end{equation}
cf.\ Vogan's \cite[Corollary 6.12]{DV}. Here, $R_{u,v}$ denotes a KLV
counterpart of the classical Kazhdan-Lusztig (KL)
$R$-polynomials.

Introducing the convenient variation $Q_{u,w}(q) =
(-q)^{\rho(w)-\rho(u)}R_{u,w}(q^{-1})$ and comparing the coefficients
of a fixed element $u\in \iota$ on each side of (\ref{eq:KLV1}) one
obtains the, from the theory of KL polynomials, familiar-looking
\begin{equation}
q^{\rho(w)}P^\iota_{u,w}(q^{-1}) = q^{\rho(u)}\sum_{v\in \iotint{u}{w}}P^\iota_{v,w}(q)Q_{u,v}(q).
\end{equation}
Together with the restrictions $P^\iota_{x,x} = 1$ and $\deg P^\iota_{u,w} \le
(\rho(w)-\rho(u)-1)/2$, this recurrence uniquely determines the KLV
polynomials. In order to use it, one must first know the
$Q$-polynomials (which {\em are} polynomials). They are completely determined by
the following recurrence and initial values; see \cite[Proposition 5.3]{AH2}.
\begin{prop}\label{pr:Qrec}
Let $u,w\in \iota$. If $s\in \DR(w)$, then
\[
Q_{u,w}(q) = 
\begin{cases} 
Q_{u\twc s,w\twc s}(q) & \text{ if } u\twc s \tl u, \\
qQ_{u\twc s,w\twc s}(q)+(q-1)Q_{u,w\twc s}(q) & \text{ if } u\twc s \tr u, \\
qQ_{u,w\twc s}(q) & \text{ if } u\twc s = u.
\end{cases}
\]
Moreover $Q_{u,u}(q)=1$, and $Q_{u,w}(q) = 0$ if $u\not \le w$.
\end{prop}

When restricted to $u,w\in \Sn\subset \iota$, both the $R_{u,w}$ and the
$Q_{u,w}$ coincide with the ordinary KL $R$-polynomials of $S_n$,
and the $P^\iota_{u,w}$ of course restrict to the ordinary KL polynomials,
i.e.\ $\KLP_{\varphi(u),\varphi(w)} = P^\iota_{u,w}$.

\section{Structural properties of $\iota$} \label{se:iota}
In this section we obtain some information about the structure of
$\iota$ which shall be of use in the sequel. 

\begin{lem}\label{onecoverinduct}
Suppose $u,u',w,w'\in \iota$ are such that $|\iotint{u}{w}| =
|\iotint{u'}{w'}| = 3$. Then, $|\iotint{u}{w}\cap\iotint{u'}{w'}|\neq 2$. 
\end{lem}
\begin{proof}
Since $\br(\twist)$ is thin, the corresponding intervals $\twint{u}{w}$ and $\twint{u'}{w'}$ have four
elements each. Hence, in $\br(\invol)$, $w_0w < x
< w_0u$ and $w_0w < y < w_0u$ where $x$ is fixed point free and
$y$ has exactly two fixed points. Consulting Incitti's
characterisation of the covering relation, this
implies that the disjoint cycle decompositions of $w_0w$, $x$ and $w_0u$
have $n-2$ two-cycles in common. Inspecting any two of these three
elements is sufficient to determine all those common
two-cycles. Since exactly three
fixed point free involutions have $n-2$ fixed two-cycles in common, we
conclude that $|\iotint{u}{w}\cap \iotint{u'}{w'}|\ge 2$ implies $\iotint{u}{w}=\iotint{u'}{w'}$.       
\end{proof}

The preceding lemma immediately yields a simple description of the twisted identities that cover exactly one element:

\begin{lem}\label{manycoat}
Let $w\in \iota$ and suppose $|\{x\in\iota\ |\ x\tl
w\}|=1$. Then, either $w=\s_{n-1}\s_n$ or $\rho(w) = 1$.
\end{lem}
\begin{proof}
It is easy to verify the assertion for all $w$ that satisfy $\rho(w)\le 2$ or
$\s_{n-1}\s_n \tl w$. Suppose $w$ is some other element and that $x\tl
w$. Applying induction on the rank, we may assume $x$ covers at least two elements. By
Lemma \ref{onecoverinduct}, so does $w$.
\end{proof}

We shall only need the following simple lemma for $w\in \iota$. Proving it
for $w\in \twist$ costs, however, no extra effort.

\begin{lem}\label{le:diag}
	Define $\tau=\s_{i+1}\s_i\s_{i-1}$, where $2\le i \le n-2$,
        and let $w\in \twist$. Then, $\tau \nleq w$ if and only if $w([i-1])\subseteq [i+1]$.
\end{lem}
\begin{proof}
Let $w\in W$ and notice that $\tau=s_{2n-(i-1)}s_{2n-i}s_{2n-(i+1)}s_{i+1}s_is_{i-1}.$ Thus,
$\DR(\tau)=\{s_{i-1},s_{2n-i-1}\}$. Following the notation used in Theorem
\ref{crit}, $(\tau_{1,i-1},\ldots,\tau_{i-1,i-1})=(1,\ldots,i-2,i+2)$
and
$(\tau_{1,2n-i-1},\ldots,\tau_{2n-i-1,2n-i-1})=(1,\ldots,2n-i-2,2n-i+2)$. Now,
$\tau \leq w$ if and only if $\tau_{k,i-1}\leq w_{k,i-1}$ for all
$1\leq k\leq i-1$ and $\tau_{k,2n-i-1}\leq w_{k,2n-i-1}$ for all
$1\leq k\leq  2n-i-1$. Therefore, $\tau \leq w$ if and only if
$\max(w(1),\cdots, w(i-1))\geq i+2$ and $\max(w(1),\cdots,
w(2n-i-1))\geq 2n-i+2$. Hence $\tau \nleq w$ if and only if
$w([i-1])\subseteq [i+1]$ or $w([2n-i-1])\subseteq [2n-i+1]$. If
$w\in \twist$, diagram symmetry yields that both inclusions are equivalent, and the result follows.
\end{proof}

\begin{lem}\label{lem2}
Let $a = \s_{i}\s_{i-1}\s_i$
and $b = \s_i\s_{i+1}\s_i$ for some $2\le i \le n-2$. If $c\in
\iota$ covers both $a$ and $b$, then $c= \s_i\s_{i-1}\s_{i+1}\s_i$.
\end{lem}
\begin{proof}
If $\s_n\le c$, the subword property shows that $c$ is obtained by
inserting the letter $\s_n$ somewhere inside some reduced $\Ss$-expression
for $a$. Since $s_n$ commutes with every generator $s_i\le a$, we may in fact
assume that $\s_n$ is inserted as the first letter. This, however, contradicts
$c \in \iota$. We conclude that $c\in \Sn$. Therefore, the
assertion of the lemma is equivalent to $c' = s_is_{i-1}s_{i+1}s_i$ being
the only element which covers $a' = s_is_{i-1}s_i$ and $b' = s_is_{i+1}s_i$ in
the (ordinary) Bruhat order on $S_n$. To see that this holds, note
that the reduced
expressions that were just used for $a'$ and $b'$ are the only ones that contain
only single occurrences of $s_{i-1}$ and $s_{i+1}$,
respectively. Hence an element which covers both must have a reduced
expression which simultaneously can be obtained by inserting $s_{i-1}$ into
$s_is_{i+1}s_i$ and by inserting $s_{i+1}$ into $s_is_{i-1}s_i$. 
\end{proof}

The final result of this section shows that a twisted identity is nearly
always determined by the elements that it covers.

\begin{prop}\label{samecoat}
Suppose $v,w\in \iota$ cover the same set of elements in
$\br(\iota)$. Then either $v=w$ or $\rho(v)=\rho(w)\leq 2$. 
\end{prop}
\begin{proof}
Assume that $v\neq w$ and that $u\tl v \Leftrightarrow u\tl w$ for
$u\in \iota$. Choose descents $s\in \DR(v)$ and $s'\in
\DR(w)$, if possible so that $v\s \neq w\s'$. 

Suppose first that, indeed, $v\s \neq w\s'$ and let $\tau = v\s'$. By the lifting
property, $w\tl \tau$. Lemma \ref{le:twistiotacover} then implies $\tau \in
\iota$. We claim that $\tau$ covers no element except
$v$ and $w$. Indeed, if $v\neq x \tl \tau$, lifting yields $x\s'
\tl v$ and thus $x\s' \tl w$. Since $w\s'$ is the
only element which does not have $s'$ as a descent among those covered
by $w$, $x\s' = w\s'$ so that $x=w$ as needed. Now, lifting shows
$v\s\s' \tl v\s$. Thus, $\tau$ has a reduced $\Ss$-expression which ends
with $\s'\s\s'$. In particular,
$\{s,s'\} = \{s_i,s_{i+1}\}$ for some $i$, and $\tau\s\s'\s =
\tau\s'\s\s'$. This means that the
$\{\s,\s'\}^*$-orbit of $\tau$ contains exactly six elements, ordered
as $\tau\s\s'\s \tl v\s,w\s' \tl v,w \tl \tau$. Assuming without loss of
generality that $s = s_i$, it follows that the disjoint cycle
decompositions of the corresponding fixed point free involutions are
as follows:
\[
\begin{split}
w_0\tau &= \trp{a}{i}\trp{b}{i+1}\trp{c}{i+2}\cdots,\\
w_0v &= \trp{a}{i}\trp{b}{i+2}\trp{c}{i+1}\cdots,\\
w_0w &= \trp{a}{i+1}\trp{b}{i}\trp{c}{i+2}\cdots,\\
w_0(v\s) &= \trp{a}{i+1}\trp{b}{i+2}\trp{c}{i}\cdots,\\
w_0(w\s') &= \trp{a}{i+2}\trp{b}{i}\trp{c}{i+1}\cdots,\\
w_0(\tau \s\s'\s)&= \trp{a}{i+2}\trp{b}{i+1}\trp{c}{i}\cdots,\\
\end{split}
\]
for some $a<b<c$; here the trailing dots indicate the remaining
two-cycles that all six elements have in common. Since conjugation by
a transposition alters either zero or two of the two-cycles of a fixed
point free involution, it follows at once from this description that
(\rm{i}) neither $v$ nor $w$ covers any element except $v\s$ and $w\s'$, and (ii) $v\s$ and
$w\s'$ cover no common element except $\tau\s\s'\s$. Lemma
\ref{onecoverinduct} then implies that $v\s$ and $w\s'$ can cover no
element at all (common or not) except $\tau\s\s'\s$. By Lemma \ref{manycoat},
$\rho(v\s) = \rho(w\s') = 1$ as desired.

It remains to consider the case $v\s = w\s' = x$. Since this
was not possible to avoid, no descent of $x$ commutes with either $s$
or $s'$. If $x\neq e$, this implies $\DR(x) = \{s_i\}$ and $\{s,s'\}=\{s_{i-1},s_{i+1}\}$ for
some $i$. If in addition $x\s_i$ has a descent, say $s_j$, it cannot commute
with $s_i$. Thus, $s_j\in \{s,s'\}$, implying that either $v$
or $w$ has $s_i$ as a descent, a contradiction. Hence $\rho(x) \le  1$.

\end{proof}

\section{Special Partial Matchings} \label{se:SPM}
Let $\Pi$ be a finite poset equipped with a unique maximum element
$\hat{1}$ and let $\prec$ denote the cover relation. 

\begin{defn}\label{de:SPM}
A {\em special partial matching}, or {\em SPM}, of $\Pi$ is a function $M:\Pi\to \Pi$ such that
\begin{itemize}
\item $M^2 = \mathrm{id}$.
\item $M(\hat{1}) \prec \hat{1}$.
\item For all $x\in \Pi$, either $M(x)\prec x$, $M(x) = x$ or $x\prec M(x)$.
\item If $x\prec y$ and $M(x)\neq y$, then $M(x) < M(y)$.
\end{itemize}
\end{defn}

An SPM without fixed points is nothing but a special
matching in the sense of Brenti \cite{FB}. Like special matchings,
SPMs restrict to principal order ideals:

\begin{prop}\label{pr:SPMrestrict}
Suppose $M$ is an SPM of $\Pi$ and that $M(x) \le x$. Then, $M$
preserves the subposet $I_x = \{y\in \Pi\mid y\le x\}$. In particular,
$M$ restricts to an SPM of $I_x$ if $M(x) \prec x$.
\end{prop}
\begin{proof}
We must show $M(y) \le x$ for all $y\le x$. Pick $y < x$ and
assume by induction $M(y') \le x$ for all $x\ge y' > y$. Choose
$x\ge z\succ y$. Then, either $M(y) = z\le x$ or $M(y) < M(z)\le x$.
\end{proof} 

Special matchings
were designed to mimic multiplication by a Coxeter generator, i.e.\
maps of the form $x\mapsto xs$, in $\br(W)$. Similarly, the idea
behind SPMs is to capture the behaviour of the twisted
conjugation maps $x\mapsto x\twc s$ in $\br(\iota)$. 

\begin{theorem}\label{th:conjmatch}
Let $w\in \iota$ and $s\in \DR(w)$. Then, $x\mapsto x\twc s$ is
an SPM of the lower interval $\iotint{e}{w}$.
\end{theorem}
\begin{proof}
The lifting property shows that $\iotint{e}{w}$ is preserved by
$x\mapsto x\twc s$. The first three properties required by Definition
\ref{de:SPM} are readily checked. It remains to verify the fourth.

Suppose $x\tl y$ and $x\twc s \neq y$. We must show $x\twc s <
y\twc s$. If $x\twc s \neq x$ and $y\twc s \neq y$, this is immediate from the
lifting property. The nontrivial cases that remain to be considered
are $x\twc s = x$, $y\twc s < y$ and $x\twc s > x$, $y\twc s = y$,
respectively. The former case is, however, impossible since it would
imply $x\s > x \neq y\s < y$ contradicting the lifting property. The
latter is in fact also impossible; it implies that $y\s$ covers the
two twisted identities $x\s$ and $y$ which contradicts Lemma \ref{le:twistiotacover}.
\end{proof}

We shall refer to an SPM of the form described in Theorem
\ref{th:conjmatch} as a {\em conjugation SPM}.

\begin{lem}\label{fix}
	Let $w\in \iota$. Suppose $M$ is an SPM of $\iotint{e}{w}$ and
         $\rho(u)\geq 2$ for $u\in\iotint{e}{w}$. Then, $M(u)=u$
        if and only if all $v\in \iota$ with $v \tl u$ satisfy either $M(v)\tl v$ or $M(v)=v$.
\end{lem}
\begin{proof}
	Suppose first that $M(u)=u$ and let $v\tl u$. Since $M$ is an
        SPM, $M(v)< M(u)=u$. Hence, $v\ntriangleleft M(v) $.

	Now assume that $v\ntriangleleft M(v) $ holds for every $v\tl
        u$. In particular, $M(u)\ntriangleleft u$. Suppose next that
        $M(u)\tr u$. By Lemma \ref{manycoat}, $u' \tl M(u)$ for some
        $u'\neq u$. Since $M(u') < M(M(u))=u$, we have $v \tl u$
        and $M(v) \tr v$ for $v = M(u')$, contradicting the hypothesis. 
	We conclude that $M(u)=u.$
	\end{proof}

Taking $w = w_1 = w_2$, the next proposition in particular shows that an SPM of $\iotint{e}{w}$ is
completely determined by its restriction to the atoms, i.e.\ the
elements that cover the identity.

\begin{prop}\label{flm}
	Let $w_1,w_2\in \iota$ and suppose $M_1$ and $M_2$ are SPMs of
        $\iotint{e}{w_1}$ and $\iotint{e}{w_2}$, respectively, such that $M_1(u)=M_2(u)$ for all
        $u\in\iotint{e}{w_1}\cap \iotint{e}{w_2}$ with $\rho(u) \le 1$. Then,
        $M_1(u)=M_2(u)$ for all $u\in \iotint{e}{w_1}\cap \iotint{e}{w_2}$.
\end{prop}
\begin{proof}
Let $M_1$ and $M_2$ satisfy the hypotheses of the proposition.
	Employing induction on the length of $u$, suppose that $\rho(u)\geq 2$ and that for all $v$ with $\rho(v)<\rho(u)$, $M_1(v)=M_2(v)$.  We consider three cases.\newline
	\textbf{Case 1:} If $M_1(u)\tl u$, then $u=M_1(M_1(u))=M_2(M_1(u))$ by
        the induction assumption. Therefore, $M_1(u)=M_2(u)$.\newline
        \textbf{Case 2:} Suppose that $M_1(u)=u$. By Lemma \ref{fix} and
        the induction hypothesis, every $v\tl u$ satisfies either
        $M_2(v)=M_1(v)=v$ or $M_2(v)=M_1(v)\tl v$. Using Lemma
        \ref{fix} again, we conclude $M_2(u)=u=M_1(u)$.\newline 
	\textbf{Case 3:} Assume now $M_1(u)\tr u$ and $M_2(u)\tr
        u$; interchanging the roles of $M_1$ and
        $M_2$ if necessary, this is the only remaining case. Let
        $A_1 = \{v \in \iota\setminus\{u\}\ | \ v\tl M_1(u)
        \}$ and $A_2 = \{v \in \iota\setminus\{u\}\ | \ v\tl M_2(u)
        \}$. It follows immediately from Definition \ref{de:SPM} that $A_i=\{M_i(x)\mid M_i(x) \tr x\tl
        u\}$. Thus, $A_1=A_2$ is implied by the induction
        assumption, so $M_1(u)$ and $M_2(u)$ cover the same set of
        elements. Since
        $\rho(M_1(u))=\rho(M_2(u))=\rho(u)+1\geq 3$, $M_1(u) = M_2(u)$ follows
        from Proposition \ref{samecoat}.
\end{proof}

\begin{cor}\label{co:conjSPM}
Suppose $M$ is an SPM of $\iotint{e}{w}$ for some $w\in \iota$. If
there exists an $s\in S$ such that $M(x) = x\twc s$ for all $x\in
\iotint{e}{w}$ with $\rho(x) \le 1$, then $M(x) = x\twc s$ for all $x\in
\iotint{e}{w}$. 
\end{cor}
\begin{proof}
The result follows from Proposition \ref{flm} if we are able to
construct $w_2\in \iota$ with $s\in \DR(w_2)$ such that $w_2 \ge w_1 =
w$. Observe that $s_i\in \DR(\hat{w})$ if and only if $i$ is even,
where $\hat{w} = w_0s_1s_3\cdots s_{2n-1}$ denotes the maximum of
$\br(\iota)$.

We may consider $W$ to be embedded in the symmetric group of
permutations of $\{0,1,\ldots,2n+1\}$ with generators $S' = S\cup
\{s_0 = \trp{0}{1},s_{2n}=\trp{2n}{2n+1}\}$ on which we have the automorphism $\theta'$ given by
$s_i \mapsto s_{2n-i}$. Then $\theta'$ restricts to $\theta$ on $W$
and $\iota$ embeds in $\iota(\theta')$. Moreover, the maximum of
$\iota(\theta')$, call it $\hat{w}'$, has $s_i$ as descent if and only if $i$ is
odd. Hence, either $w_2 = \hat{w}$ or $w_2= \hat{w}'$ does the job.
\end{proof}

The recurrence relation in Proposition \ref{pr:Qrec} relies on a
conjugation SPM. Our goal is to replace it with an arbitrary SPM in
order to arrive at a combinatorially defined recurrence for the
$Q$-polynomials. In order to do so, we need a good understanding of
non-conjugation SPM behaviour. The next lemma imposes strong restrictions on
such partial matchings.

\begin{lem}\label{M2}
Suppose $M$ is an SPM of $\iotint{e}{w}$ for $w\in \iota$. If
$M$ is not a conjugation SPM, then $M(e)=\s_i$ for some $2\le i\le n-2$ and one
of the following two sets of conditions holds:
\begin{enumerate}
\item $M(\s_{i-1}) = \s_{i-1}\s_i$, $M(\s_{i+1}) = \s_i\s_{i+1}$ and
  $\s_{i+1}\s_i\s_{i-1} \not \le w$. 
\item $M(\s_{i-1}) = \s_i\s_{i-1}$, $M(\s_{i+1}) = \s_{i+1}\s_i$ and
  $\s_{i-1}\s_i\s_{i+1} \not \le w$. 
\end{enumerate}
\end{lem}
\begin{proof}
Let $M$ be an SPM of $\iotint{e}{w}$. 

First, consider the case
$M(e)=e$. We cannot have $M(\s_i)=\s_i$ for every $\s_i\le w$ because Lemma
\ref{fix} would then imply that $M(u) = u$ for all $u\le w$
contradicting that $M$ is an SPM. Hence $M(\s_i)=\s_i\s_j$ for some
$i$ and $j$. Then $x\tl \s_i\s_j$, $x\in \iota$, can only happen if $x = \s_i$; otherwise
$M(x)\tl x$ which is impossible. By Lemma \ref{manycoat}, this means $\s_i = \s_{n-1} = \s_{n+1}$ and
$j=n$. In other words, $M(u) = u\twc s_n$ for all $u\in \iotint{e}{w}$ with
$\rho(u)\le 1$, whence $M$ is a conjugation SPM by Corollary \ref{co:conjSPM}.    

Second, suppose $M(e) = \s_i$, $i<n$. For $i\neq j<n$, the only
elements which cover both $\s_i$ and $\s_j$ are $\s_i\s_j =
\s_j\s_{2n-i}$ and $\s_j\s_i = \s_i\s_{2n-j}$; they coincide if and
only if $i\neq j\pm 1$. If $M$ is neither twisted conjugation by $s_i$
nor by $s_{2n-i}$ we must therefore have $M(\s_j) = \s_i\s_j \neq \s_j\s_i$ and
$M(\s_k) = \s_k\s_i \neq \s_i\s_k$ for some $j,k<n$, $\s_j,\s_k\le w$. This is
only possible if $\{j,k\} = \{i-1,i+1\}$. In
particular, $2\le i\le n-2$. Assume $M(\s_{i-1})=\s_{i-1}\s_i$ and
$M(\s_{i+1})=\s_i\s_{i+1}$, the other case being entirely
similar. Suppose in order to obtain a contradiction
$\s_{i+1}\s_i\s_{i-1}\le w$. By the subword property,
$\s_{i+1}\s_i\leq w$ and $\s_i\s_{i-1}\leq w$. Since $M(\s_{i+1}\s_i)$
covers both $\s_{i+1}\s_i$ and $\s_i\s_{i+1}$,
$M(\s_{i+1}\s_i)=\s_i\s_{i+1}\s_i$. Similarly,
$M(\s_i\s_{i-1})=\s_i\s_{i-1}\s_i$. Now, $M(\s_{i+1}\s_i\s_{i-1})$
covers $\s_i\s_{i+1}\s_i$ and $\s_i\s_{i-1}\s_i$, so
$M(\s_{i+1}\s_i\s_{i-1}) = \s_i\s_{i-1}\s_{i+1}\s_i$ by Lemma
\ref{lem2}. This is however impossible since $\s_{i+1}\s_i\s_{i-1}\nleq \s_i\s_{i-1}\s_{i+1}\s_i$.
\end{proof}

A conjugation SPM may have fixed points. The upcoming
proposition, however, states that any non-conjugation
SPM is fixed point free and commutes with some fixed point free
conjugation SPM.

\begin{prop} \label{com}
Let $w\in \iota$ and assume $M$ is an SPM of $\iotint{e}{w}$ which is not a
conjugation SPM. Then $M$ has no fixed point. Moreover, there
exists $s\in D_{R}(w)$ such that $w\twc s \neq M(w)$ and, furthermore, $u\twc s\neq u$ and $M(u\twc s) =
M(u)\twc s$ for all $u\in \iotint{e}{w}$.
\end{prop}
\begin{proof}
Lemma \ref{M2} shows that if $M$ is a non-conjugation SPM of $\iotint{e}{w}$, $M(e) =
\s_i$ for some $2\le i \le n-2$, and either $M(\s_{i+1}) = \s_i\s_{i+1}$ and $M(\s_{i-1}) =
\s_{i-1}\s_i$ or else $M(\s_{i+1}) = \s_{i+1}\s_i$ and $M(\s_{i-1}) =
\s_i\s_{i-1}$. Replacing $w$ with $w^{-1}$ if necessary, we may
assume the former situation is at hand. It follows that
$\s_{i+1}\s_i\s_{i-1} \not \le w$. Hence, Lemma \ref{le:diag} implies
that every element in $\iotint{e}{w}$ has a permutation diagram of the form
illustrated in Figure \ref{fi:diag}.

\begin{figure}[htb]
\includegraphics[scale=0.5]{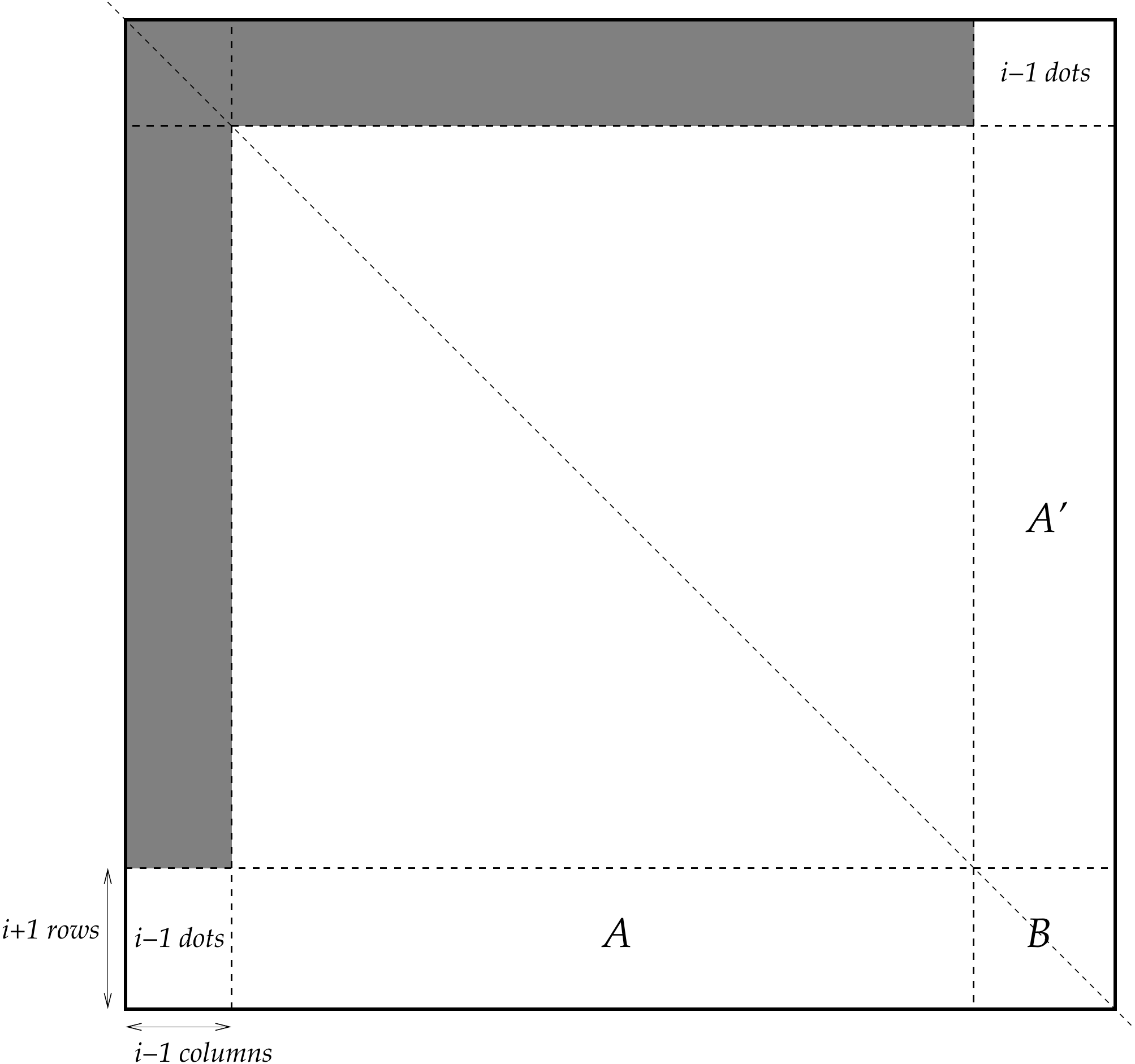}
\caption{Illustration for the proof of Proposition
  \ref{com}. Permutation diagrams of elements in $\iotint{e}{w}$ have
  the depicted form. Shaded regions are empty.} \label{fi:diag}
\end{figure}

For any $u\in \iotint{e}{w}$, its diagram either contains two dots in each of $A$
and $A'$, and $B$ is empty, or else $B$ contains two dots and both $A$ and $A'$ are empty. Say $u$ is of {\em type
  $4$} in the former case and {\em type $2$} in the latter; i.e.\ the
type indicates the total number of dots in $A\cup A'\cup B$.

Recall the description of $\tl$ from Figure \ref{fi:incitti}. Let us
say that a covering of the form depicted in the lower left picture is
produced by a {\em box cover transformation} which {\em involves} the dots which are indicated in
the picture, i.e.\ those not shared by the two diagrams. 

\smallskip

\noindent {\bf Claim.} If $u$ is of type $4$, $M(u)$ is obtained from $u$ by a
box cover transformation which involves the four dots in $A\cup
A'$, whereas $M(u) = u$ if $u$ is of type $2$.

\smallskip

The claim is readily verified if $\rho(u)\le 1$. In order to prove it in general,
we assume $\rho(u)\ge 2$ and induct on $\rho(u)$. 

Suppose first that $u$ is of type $2$. Recalling from Figure \ref{fi:incitti} the description of
the cover relation in $\iota$, it is clear that $v\tl u$
implies $v$ is of type $2$ or of type $4$ with the dots in $A'$
forming a fall. The induction assumption shows that $M(v) = v$ (in the
former case) or $M(v) \tl v$ (in the latter). By Lemma \ref{fix}, $M(u)
= u$ as desired.

Now assume $u$ is of type $4$. Let $b$ denote the operator which
acts on elements of type $4$ by applying a box cover transformation
involving the dots in $A$ and $A'$. The induction assumption
implies $b(v)=M(v)$ for all $u\neq v \in \iotint{e}{u}$ since all such $v$ are of type $4$ by
Theorem \ref{th:stdcrit}. It must be shown that $b(u)=M(u)$. If $b(u)<u$
or $M(u) < u$ we are done by induction, so suppose $b(u) > u$ and
$M(u) \ge u$. In order to obtain a contradiction, assume $M(u) \neq
b(u)$. It suffices to find $x\tl M(u)$, $x\neq u$, such that $b(x)>
x$; since $M$ is an SPM it would satisfy $M(x)\tl x$ (if $M(u) \tr u$) or
$M(x) \le x \tl u$ (if $M(u) = u$), so that the induction assumption
implies $b(x) = M(x)$ which is the needed contradiction.

Consider the diagram of $M(u)$ as depicted in Figure \ref{fi:diag}. If $b(M(u))
< M(u)$, $x=b(M(u))$ has the desired properties. Hence, we may assume
the dots in $A$ form a rise, as do those in $A'$. By
Proposition \ref{pr:SPMrestrict}, $M(u)\ge \s_i = M(e)$. This implies
$\s_{i-1}\le M(u)$ or $\s_{i+1}\le M(u)$, since otherwise $s_i$
would be a descent of $M(u)$ corresponding to a fall in $A$.

First, if $\s_{i-1}\le M(u)$, the dots in $A'$ are not in the two leftmost
columns by Theorem \ref{th:stdcrit}. Therefore there exists $k\in [i]$ such that $\theta(s_k)\in
\DR(M(u))$, and this descent involves exactly one dot in $A'$ (by
which we mean $M(u)(2n-k)>2n+1-i$ and $M(u)(2n+1-k)\le 2n+1-i$). It is
clear that $b(M(u)\twc \theta(s_k)) > M(u)\twc \theta(s_k)$.

Second, if $\s_{i+1}\le M(u)$, the dots in $A$ are not in the two leftmost
columns. Then, $s_j\in \DR(M(u))$ for some $i\le j \le 2n-i-2$ with
this descent involving exactly one dot in $A$ (i.e.\ $M(u)(j)>i+1$ and
$M(u)(j+1)\le i+1$), implying $b(M(u)\twc s_j) > M(u)\twc s_j$.

Now, if both $\s_{i-1}\le M(u)$ and $\s_{i+1}\le M(u)$, we note $M(u)\twc
\theta(s_k) \neq M(u)\twc s_j$ e.g.\ since the latter element coincides with
$M(u)$ on $[i-1]$ whereas the former does not. Thus, at least one of
them is not equal to $u$; let $x$ be this element.

Finally, let us consider the case $\s_{i+1}\le M(u)$, $\s_{i-1}\not\le
M(u)$ (the situation $\s_{i+1}\not\le M(u)$, $\s_{i-1}\le
M(u)$ being completely analogous). Let $x=M(u)\twc s_j$. If $M(u) =
u$, $x\neq u$ and we are done. If $u < M(u)$, $M$ restricts to an SPM
of $\iotint{e}{M(u)}$ by Proposition \ref{pr:SPMrestrict}.
  Since $\s_{i-1}\not\le M(u)$, Lemma \ref{M2} implies that the
  restriction is a conjugation SPM. Examining $M(e)$ and $M(s_{i+1})$,
  we conclude $M(v) = v\twc \theta(s_i)$ for all $v\in
  \iotint{e}{M(u)}$. Hence, $u = M(u) \twc \theta(s_i) \neq x$. The
  claim is established.

Now, Theorem \ref{th:stdcrit} implies that $w$ is of type $2$ whenever
$u$ is of type $2$ for some $u\le w$. Since $M(w) \neq w$ it follows from
the claim that every element in $\iotint{e}{w}$ must in fact be of type $4$. In
particular, $M$ coincides with $b$ which has no fixed point.

Just as in the proof of the claim, the fact that $\s_{i-1} \le w$ implies
$\theta(s_k)\in \DR(w)$ for some $i\in [k]$ with $w\twc
\theta(s_k) \neq b(w)$. Moreover, for $u \in \iotint{e}{w}$,
$u\twc \theta(s_k) = u$ would imply $u$ is of type $2$ which we have
just seen is impossible. Finally, it is not hard to see that $b(u)\twc
s = b(u\twc s)$ for any $s\in S$. In particular, $s=\theta(s_k)$ has
all the asserted properties. 
\smallskip
\end{proof}

\section{KLV polynomials}
Finally, we have gathered all ingredients that are necessary in order to prove the main result which
asserts that any SPM of $\iotint{e}{w}$ can be used in the recurrence
relation for
the $Q$-polynomials of the intervals $\iotint{u}{w}$. With the key SPM
properties from the previous section under the belt, the arguments
that remain are essentially identical to those employed by
Brenti in his proof of \cite[Theorem 5.2]{FB}.
\begin{theorem}\label{res}
Let $M$ be an SPM of $\iotint{e}{w}$. Then, for any $u\in \iotint{e}{w}$,
\[ 
Q_{u,w}(q)= 
\begin{cases} 
Q_{M(u),M(w)}(q) & \text{ if } M(u)\tl u, \\
qQ_{M(u),M(w)}(q)+(q-1)Q_{u,M(w)}(q) & \text{ if } M(u)\tr u, \\
qQ_{u,M(w)}(q) & \text{ if } M(u)= u.
\end{cases}
\]
\end{theorem}

\begin{proof}
This is just Proposition \ref{pr:Qrec} if $M$ is a conjugation SPM, so suppose it is not; in
particular $\rho(w)\ge 3$ by Lemma \ref{M2}. We induct
on $\rho(w)$. By Proposition
        \ref{com}, there exists $s\in \DR(w)$, $M(w)\neq w\twc s$, such that the conjugation SPM given by $s$ commutes
        with $M$ and fixes no element in $\iotint{e}{w}$. Let $u\in\iotint{e}{w}$. Proposition \ref{com} shows $M(u)\neq u$, so we consider two cases: \newline 
	\textbf{Case 1:} $M(u)\tl u$. We need to show that
        $Q_{u,w}(q)=Q_{M(u),M(w)}(q)$. \newline  Suppose first that
        $s\in \DR(u)$. Then, either $s\in \DR(M(u))$ or $M(u)=u\twc
        s$. If $s\in \DR(M(u))$, $M(u\twc s)=M(u)\twc s\tl u\twc
        s$. Therefore,
\[
Q_{u,w}(q)=Q_{u\twc s,w\twc s}(q)=Q_{M(u\twc s),M(w\twc
  s)}(q)=Q_{M(u)\twc s,M(w)\twc s}(q)=Q_{M(u),M(w)}(q),
\]
where the second equality (as is
the case in all subsequent computations of this kind throughout the
proof) follows from the inductive hypothesis and
the fact, provided by Proposition \ref{pr:SPMrestrict}, that $M$ is an SPM of $\iotint{e}{w\twc s}$. If $M(u)=u\twc s$, then 
\begin{eqnarray*}
	Q_{u,w}(q)&=&Q_{u\twc s,w\twc s}(q)\\
	&=&qQ_{M(u\twc s),M(w\twc s)}(q)+(q-1)Q_{u\twc s,M(w\twc s)}(q)\\
	&=&qQ_{M(u)\twc s,M(w)\twc s}(q)+(q-1)Q_{M(u),M(w)\twc s}(q)\\
	&=&Q_{M(u),M(w)}(q).
\end{eqnarray*}
Suppose now that $s\notin \DR(u)$. Since twisted conjugation by $s$
does not fix any element in $\iotint{e}{w}$, $M(u) \tl u \tl u\twc s$. We also
have $M(u) \tl M(u)\twc s = M(u\twc s) \tl u\twc s$ because
$\rho(u\twc s) = \rho(M(u))+2$. Therefore, 
\begin{eqnarray*}
	Q_{u,w}(q)&=&qQ_{u\twc s,w\twc s}(q)+(q-1)Q_{u,w\twc s}(q)\\
	&=&qQ_{M(u\twc s),M(w\twc s)}(q)+(q-1)Q_{M(u),M(w\twc s)}(q)\\
	&=&qQ_{M(u)\twc s,M(w)\twc s}(q)+(q-1)Q_{M(u),M(w)\twc s}(q)\\
	&=&Q_{M(u),M(w)}(q).
\end{eqnarray*}
\textbf{Case 2:} $M(u)\tr u$. We must prove that $Q_{u,w}(q)=qQ_{M(u),M(w)}(q)+(q-1)Q_{u,M(w)}(q)$.\newline 
First, assume that $s\in \DR(u)$. Then we have $u\twc s \tl u \tl
M(u)$, and therefore also $u\twc s\tl M(u\twc s) = M(u)\twc s \tl M(u)$. Hence, 
\begin{eqnarray*}
	Q_{u,w}(q)&=&Q_{u \twc s,w\twc s}(q)\\
	&=&qQ_{M(u\twc s),M(w\twc s)}(q)+(q-1)Q_{u\twc s,M(w\twc s)}(q)\\
	&=&qQ_{M(u)\twc s,M(w)\twc s}(q)+(q-1)Q_{u\twc s,M(w)\twc s}(q)\\
	&=&qQ_{M(u),M(w)}(q)+(q-1)Q_{u,M(w)}(q).\\
\end{eqnarray*}
Suppose now that $s\notin \DR(u)$. Again, this means $u\tl u\twc
s$. If $M(u)\neq u\twc s$, then $u\tl M(u) \tl M(u)\twc s=M(u\twc s)$.
Hence, $u\twc s \tl M(u\twc s)$. We obtain   
\begin{eqnarray*}
	Q_{u,w}(q)&=&qQ_{u\twc s,w\twc s}(q)+(q-1)Q_{u,w\twc s}(q)\\
	&=&q(qQ_{M(u\twc s),M(w\twc s)}(q)+(q-1)Q_{u\twc s,M(w\twc
          s)}(q))\\
        & & +(q-1)(qQ_{M(u),M(w\twc s)}(q) +(q-1)Q_{u,M(w\twc s)}(q))\\
	&=&q(qQ_{M(u)\twc s,M(w)\twc s}(q)+(q-1)Q_{M(u),M(w)\twc
          s}(q))\\
        & &+(q-1)(qQ_{u\twc s,M(w)\twc s}(q) +(q-1)Q_{u,M(w)\twc s}(q))\\
	&=&qQ_{M(u),M(w)}(q)+(q-1)Q_{u,M(w)}(q).\\
\end{eqnarray*}
Finally, if $M(u)=u\twc s$, then, 
\begin{eqnarray*}
	Q_{u,w}(q)&=&qQ_{u\twc s,w\twc s}(q)+(q-1)Q_{u,w\twc s}(q)\\
	&=&qQ_{M(u\twc s),M(w\twc s)}(q)+(q-1)(qQ_{M(u),M(w\twc
          s)}(q)+(q-1)Q_{u,M(w\twc s)}(q))\\
	&=&qQ_{M(u)\twc s,M(w)\twc s}(q)+(q-1)(qQ_{u\twc s,M(w)\twc
          s}(q)+(q-1)Q_{u,M(w)\twc s}(q))\\
	&=&qQ_{M(u),M(w)}(q)+(q-1)Q_{u,M(w)}(q).\\
\end{eqnarray*}
\end{proof}

Since $Q_{u,w}$ is determined by the SPMs of the intervals
$\iotint{e}{v}$ for $v\in \iotint{e}{w}$, this also holds for the KLV
$R$-polynomials and the KLV polynomials themselves. Since an SPM is a
poset invariant, the next corollary follows.

\begin{cor}
If $f:\iotint{e}{w}\rightarrow \iotint{e}{w'}$ is a poset isomorphism, then for all $v\in \iotint{e}{w}$, $Q_{v,w}= Q_{f(v),w'},$
$R_{v,w}= R_{f(v),w'}$ and 
$P^\iota_{v,w}= P^\iota_{f(v),w'}$.
\end{cor}

In particular, Theorem \ref{th:mainv1} is established.

\end{document}